\documentclass[12pt, a4]{amsart}
\usepackage{geometry}
\usepackage{amsxtra,amssymb,amsthm,amsmath,amscd,mathrsfs, epsfig, eufrak}
\usepackage{amscd, amsmath, mathrsfs, amssymb, amsthm, amsxtra, bbding, epsfig, eucal, eufrak, graphicx, latexsym, mathrsfs, mathbbol, bbold}
\usepackage[all]{xy}

\usepackage{tikz-cd}
\usepackage{hyperref}
\hypersetup{hypertex=true,colorlinks=true,linkcolor=cyan,anchorcolor=cyan,citecolor=cyan}
\usepackage{tikz}
\usepackage[normalem]{ulem}
\usepackage{soul}
\usepackage{color}

\setstcolor{red}
\newcommand{\blue}[1]{{\color{blue}#1}}
\setstcolor{blue}
\makeatletter
\@namedef{subjclassname@2020}{%
	\textup{2020} Mathematics Subject Classification}
\makeatother

\def \C {{\mathbb C}}

\def \N {{\mathbb N}}

\def \d {\,{\rm d}}
\def\re{{\Re e\,}}

\def \sset {{\smallsetminus }}

\def\geq{\geqslant}
\def\le{\leqslant}
\def\ge{\geqslant}

\theoremstyle{plain}
\newtheorem{theorem}{Theorem}[section]
\newtheorem{proposition}{Proposition}[section]
\newtheorem{lemma}[proposition]{Lemma}

\theoremstyle{remark}

\numberwithin{equation}{section}
\addtocounter{footnote}{1}

\numberwithin{equation}{section}

\addtocounter{footnote}{1}

\begin{document}
	\title[\tiny On the distribution of large values of  $|\zeta(1+{\rm i}t)|$]
	{On the distribution of large values of  $|\zeta(1+{\rm i}t)|$}
	\author[\tiny Zikang Dong]{Zikang Dong}
	\address{%
		CNRS LAMA 8050\\
		Laboratoire d'analyse et de math\'ematiques appliqu\'ees\\
		Universit\'e Paris-Est Cr\'eteil\\
		61 avenue du G\'en\'eral de Gaulle\\
		94010 Cr\'eteil Cedex\\
		France
	}
	\email{zikangdong@gmail.com}

	\date{\today}
	
	
	\begin {abstract}
	In this article, we study the distribution of large values of the Riemann zeta function on the 1-line.
	We obtain an improved density function concerning large values, holding in the same range as that given by Granville and Soundararajan.

\end{abstract}

\maketitle

\section {\color {blue} \large Introduction}

The study of the value distribution of the Riemann zeta function $\zeta(s)$ can date back to the early twentieth century when Bohr showed that for any $z\in\C^*$ and $\varepsilon>0$, 
there are infinitely many $s$'s with $1<\re s<1+\varepsilon$ such that $\zeta(s)=z$. Later in 1932, Bohr and Jessen \cite{BohrJessen} showed that $\log\zeta(\sigma+{\rm i} t)$ has a continuous distribution on the complex plane for any $\sigma>\frac{1}{2}$.

In this article, we focus on the values of $\zeta(s)$ on the 1-line (i.e. the right borne of the crtical strip). 
The values on the 1-line have much significance. 
For example, the fact that $\zeta(1+{\rm i} t)\neq0$ implies the prime number theorem. The extreme values of $\zeta(1+{\rm i} t)$ has been widely investigated. In 1925, Littlewood \cite{Li25} showed that there exist arbitrarily large $t$ for which
\begin{align*}
	|\zeta(1+{\rm i}t)|\geq \{1+o(1)\}{\rm e}^\gamma\log_2t .
\end{align*}
Here and throughout, we denote by $\gamma$ the Euler constant and by $\log_j$ the $j$-th iterated logarithm.
In 1972, Levinson \cite{Le72} improved the error term from $o(\log_2t)$ to $O(1)$:
$$
|\zeta(1+{\rm i}t)|\geq {\rm e}^\gamma\log_2t +O(1).
$$
In 2006, Granville and Soundararajan \cite{GS06} got a much stronger result
$$
\max_{t\in[1,T]}|\zeta(1+{\rm i}t)|\geq {\rm e}^\gamma \{\log_2T+\log_3T-\log_4T+O(1)\}
$$
holds for sufficiently large $T$. 
Then in 2019, Aistleitner, Mahatab, and Munsch \cite{Ai19} canceled the term $\log_4T$:
\begin{equation}\label{extreme}
\max_{t\in[\sqrt T,T]}|\zeta(1+{\rm i}t)|
\geq {\rm e}^\gamma \{\log_2T+\log_3T+O(1)\}.
\end{equation}
This bound is best possible up to the error term $O(1)$, since in \cite{GS06}, Granville and Soundararajan conjectured that
\begin{equation}\label{conj}
\max_{t\in[T,2T]}|\zeta(1+{\rm i}t)|
= {\rm e}^\gamma \{\log_2T+\log_3T+C_0+1-\log2+o(1)\},
\end{equation}
where $C_0$ is some absolute constant (see \eqref{def:Cj} below). 
This conjecture was based on some analysis on the following distribution function they introduced in \cite{GS06}:
 define for $T>1$,
\begin{equation}\label{def:PhiTtau}
	\Phi_T(\tau)
	:= \frac{1}{T} {\rm meas}\big\{t\in[T,2T]:|\zeta(1+{\rm i}t)|>{\rm e}^{\gamma}\tau\big\}.
\end{equation}
Then they proved the asymptotic formula in the logarithm of the distribution function
\begin{align}\label{distribution}
	\Phi_T(\tau)=\exp\bigg(\!-\frac{2{\rm e}^{\tau-C_0-1}}{\tau}
	\bigg\{1
	+ O\bigg(\frac{1}{\sqrt\tau}+\sqrt{\frac{{\rm e}^{\tau}}{\log T}}\bigg)\bigg\}\bigg)
\end{align}
valid uniformly for $1\le \tau<\log_2T-20$, where
\begin{equation}\label{def:Cj}
C_j 
:= \int_0^2 \bigg(\log\frac{t}{2}\bigg)^j \frac {\log I_0(t)}{t^2} \d t+\int_2^\infty \bigg(\log\frac{t}{2}\bigg)^j
\frac {\log I_0(t)-t}{t^2} \d t	
\end{equation}
and
\begin{equation}\label{def:I0t}
I_0(t) := \sum_{n=0}^{\infty} \frac{(t/2)^n}{(n!)^2}\cdot
\end{equation}
The distribution function (\ref{distribution}) describes the frequency with which each large value is attained. Obviously, the maximum of the range of $\tau$ is much less than the large value (\ref{extreme}). However, if (\ref{distribution}) were to persist to the end of the viable range, then we could get (\ref{conj}). 

The method to prove (\ref{distribution}) was also ajusted to apply to the distribution of values on the 1-line of other $L$-functions. Also in \cite{GS06}, Granville and Soundararajan showed that the distribution of the Dirichlet $L$-functions in the character-aspect has the same form as (\ref{distribution}). This result can be used to study the distribution of large character sums, see \cite{BG13} and \cite{BGGK}. 
In 2003, Granville and Soundararajana \cite{GS03} established the distribution of the Dirichlet $L$-functions 
of quadratic characters $L(1,\chi_d)$, which proves part of Montgomery and Vaughan's conjecture in \cite{Mon77}. In 2007, Wu \cite{Wu2007} improved this result by giving a high order expansion in the exponent of the distribution function. 
In 2008, Liu, Royer and Wu \cite{LRW2008} gave the distribution of a kind of symetric power $L$-functions. 
In 2010, Lamzouri \cite{La2010} studied a generalized $L$-function 
which can cover the results of \cite{GS03, LRW2008}. 
Again concerning the Riemann zeta function, in 2008 Lamzouri \cite{La2008} generalize (\ref{distribution}) to the joint distribution of $\arg\zeta(1+{\rm i} t)$ and $|\zeta(1+{\rm i} t)|$. 

Inspired by the result of Wu \cite{Wu2007}, the aim of this paper is to get an improvement of (\ref{distribution}), 
which presents a higher order expansion in the exponent.
\begin{theorem}\label{thm1}
	There is a sequence of real numbers $\{\mathfrak{a}_j\}_{j\ge 1}$ such that for any integer $J\ge 1$ we have
	\begin{equation*}
		\Phi_T(\tau)=\exp\bigg(\!-\frac{2{\rm e}^{\tau-C_0-1}}{\tau}
		\bigg\{1+\sum_{j=1}^J \frac{\mathfrak{a}_j}{\tau^j}
		+ O_J\bigg(\frac{1}{\tau^{J+1}}+\sqrt{\frac{{\rm e}^{\tau}}{\log T}}\bigg)\bigg\}\bigg)
	\end{equation*}
	uniformly for $T\to\infty$ and $1\le \tau\le\log_2T-20$,
	where $C_0$ is defined as in \eqref{def:Cj}. Moreover, $ \mathfrak{a}_1=2(1+C_0 - C_1)$.
\end{theorem}

Our main new ingredient for the proof of Theorem \ref{thm1} is Proposition \ref{prop5.1} below,
which gives a better approximation of the distribution function of the short Euler products:
\begin{equation}\label{def:PhiTtauy}
\Phi_T(\tau; y) := \frac{1}{T}  {\rm meas}\big\{t\in[T,2T]:|\zeta(1+{\rm i}t; y)|>{\rm e}^{\gamma}\tau\big\},
\end{equation}	
where
\begin{equation}\label{ded:zetasy}
\zeta(s; y) := \prod_{p\le y} (1-p^{-s})^{-1}.
\end{equation}
For this, it is necessary to improve Theorem 3 of \cite{GS06} 
(see Propositions \ref{prop3.1} and \ref{prop4.1} below).

\vskip 6mm

\section{Preliminary lemmas}

Let $k\ge 1$ be a positive integer. Define $d_k(n)$ by the relation
\begin{equation}\label{def:dk(n)}
\zeta(s)^k = \sum_{n\ge 1} d_k(n)n^{-s}
\qquad
(\re s>1).
\end{equation}
Firstly, we quote the following asymptotic formulae of sums attached to the divisor function $d_k(n)$
and the Bessel function $I_0(t)$ to show their correlation.
	
    	\begin{lemma}\label{lem2.1}
		For any prime $p$ and positive integer $k$, we have
		\begin{align}
			\sum_{\nu\ge 0} \frac{d_k(p^{\nu})^2}{p^{2\nu}}
			& = I_0\bigg(\frac{2k}p\bigg) \exp\bigg\{O\bigg(\frac k{p^2}\bigg)\bigg\},
			\label{lem2.1:1}
			\\\noalign{\vskip 0,5mm}
			\frac{\min(1, \, p/k)}{50} \bigg(1-\frac{1}{p}\bigg)^{-2k}
			& \le \sum_{\nu\ge 0}\frac{d_k(p^{\nu})^2}{p^{2\nu}}
			\le \bigg(1-\frac{1}{p}\bigg)^{-2k},
			\label{lem2.1:2}
		\end{align}
		where $I_0(t)$ is the Bessel function as defined in \eqref{def:I0t}.
	\end{lemma}
	
	\begin{proof}
		This is \cite[Lemma 4]{GS06}.
	\end{proof}
In order to approximate the Riemann zeta function $\zeta(s)$ by its truncated Euler product $\zeta(s; y)$ 
defined by \eqref{ded:zetasy}, we need the following evaluation for moments of the sum over complex power of primes between two large numbers.

\begin{lemma}\label{lem2.2}
Let $\{b(p)\}_{p\;{\rm primes}}$ be a complex sequence. Then we have
$$
\frac1 T\int_T^{2T} \bigg|\sum_{y\le p\le z} \frac{b(p)}{p^{{\rm i}t}}\bigg|^{2k}\d t
\ll \bigg(k\sum_{y\le p\le z} |b(p)|^2\bigg)^k + \frac{1}{T^{2/3}}\bigg(\sum_{y\le p\le z} |b(p)|\bigg)^{2k}
$$
uniformly for $T\ge 8$, $2\le y\le z\le T^{1/3}$ and all integers $1\le k\le\log T/(3\log z)$,
where the implied constant is absolute.
\end{lemma}
	
	\begin{proof}
		This is \cite[Lemma 3]{GS06}.
	\end{proof}

The following lemma is an approximation of $\zeta(s)$ by $\zeta(s; y)$.

\begin{lemma}\label{lem3}
		Let $T\ge 2$ and $y\ge \log T$. Then we have
		$$
		\zeta(1+{\rm i}t)
		= \zeta(1+{\rm i}t; y) \bigg\{1+O\bigg(\frac{\sqrt{(\log T)/y}}{\log_2T}\bigg)\bigg\}
		$$
		for all $t\in [T, 2T]$ except for a set of measure at most $O(T\exp\{-(\log T)/(50\log_2T)\})$.
	\end{lemma}
	
\begin{proof}
This is essentially \cite[Proposition 1]{GS06} while we erase restrictions of the upper bound of $y$.
In fact, the truncated Euler product with larger length would provide a better approximation of the zeta function.
\end{proof}

\vskip 6mm

\section{An asymptotic developement}

The integer $n\ge 1$ is called $y$-friable if the largest prime factor $P(n)$ of $n$ is less than $y$
($P(1)=1$ by convention).
Denote by $S(y)$ the set of $y$-friable integers and define
\begin{equation}\label{def:Dky}
D_k(y) := \sum_{n\in S(y)} d_k(n)^2/n^2.
\end{equation}
The aim of this section is to prove the following proposition, 
which is an improvement of the second part of Theorem 3 in \cite{GS06}.
Our improvement is double: 
a higher order expansion in the exponent and a larger domain of $y$.
	
	\begin{proposition}\label{prop3.1}
		Let $A>1$ be a positive number and let $J\ge 0$ be an integer.
		We have
		$$
		D_k(y)
		= \prod_{p\le k}\bigg(1-\frac{1}{p}\bigg)^{-2k}
		\exp\bigg(\frac{2k}{\log k}\bigg\{\sum_{j=0}^J \frac{C_j}{(\log k)^j}
		+ O_{A, J}\bigg(\frac1{(\log k)^{J+1}}+\frac {k} {y}\bigg)\bigg\}\bigg)
		$$
		uniformly for $k\ge 2$ and $2k\le y\le k^A$,
		where the $C_j$ is defined as in \eqref{def:Cj}
and the constant implied depends on $A$ and $J$ only.
	\end{proposition}

\begin{proof}
Firstly we have trivially
\begin{equation}\label{proof-th1:2}
\prod_{\sqrt{k}<p\le y} \exp\bigg\{O\bigg(\frac k{p^2}\bigg)\bigg\}
= {\rm e}^{O(\sqrt{k})},
\qquad
\prod_{p\le \sqrt{k}} \frac{\min(1, \, p/k)}{50}
= {\rm e}^{O(\sqrt{k})}.
\end{equation}
Secondly, by Lemma \ref{lem2.1}, we can write
\begin{equation}\label{proof-th1:3}
D_k(y)
= \prod_{p\le \sqrt{k}} \bigg(1-\frac{1}{p}\bigg)^{-2k}
\prod_{\sqrt{k}<p\le y}I_0\bigg(\frac{2k}p\bigg) {\rm e}^{O(\sqrt{k})}
= \prod_{p\le k} \bigg(1-\frac{1}{p}\bigg)^{-2k}
\Pi_1 \Pi_2 \, {\rm e}^{O(\sqrt{k})},
\end{equation}
where
$$
\Pi_1
:= \prod_{\sqrt{k}<p\le k} \bigg(1-\frac{1}{p}\bigg)^{2k}I_0\bigg(\frac{2k}p\bigg)
\qquad\text{and}\qquad
\Pi_2
:= \prod_{k<p\le y}I_0\bigg(\frac{2k}p\bigg).
$$
In order to evaluate $\Pi_1$, 
we apply the formula $\log(1+t)=t+O(t^2) \; (|t|\le \tfrac{1}{2})$ and the first estimate in \eqref{proof-th1:2} to obtain 
$$
\Pi_1 
= \exp\bigg\{\sum_{\sqrt{k}<p\le k} \bigg(\log I_0\bigg(\frac{2k}{p}\bigg) - \frac{2k}{p}\bigg)
	+ O\big(\sqrt{k}\big)\bigg\}.
$$
Recall the prime number theorem
$$\pi(u):=\sum_{p\le u} 1 = \int_2^u \frac{\d u}{\log u} + O(u{\rm e}^{-2c\sqrt{\log u}}).$$
Then we can derive that
\begin{equation}\label{3.5}
\begin{aligned}
\sum_{\sqrt{k}<p\le k} \bigg(\log I_0\bigg(\frac{2k}{p}\bigg) - \frac{2k}{p}\bigg)
	& = \int_{\sqrt{k}}^k \bigg(\log I_0\bigg(\frac{2k}{u}\bigg)-\frac{2k}{u}\bigg) \d\pi(u)
	\\
	& = \int_{\sqrt{k}}^k\frac{\log I_0(2k/u)-2k/u}{\log u} \d u + O\big(k{\rm e}^{-c\sqrt{\log k}}\big).
\end{aligned}
\end{equation}
By putting $t=2k/u$ and using the fact that
$$\frac{1}{1-t}
=\sum_{j=0}^J t^j + O_J(t^{J+1}) 
\qquad 
(|t|\le \tfrac{1}{2}),$$
we can derive that the integral in \eqref{3.5} is equal to
\begin{align}\label{3.6}
\frac{2k}{\log k} \int_2^{2\sqrt{k}}\frac{\log I_0(t)-t}{t^2 (1-\frac{\log(t/2)}{\log k})} \d t
= \frac{2k}{\log k} \bigg\{\sum_{j=0}^{J} \frac{C^{*}_j(k)}{(\log k)^j}
	+ O_J\bigg(\frac{C^{*}_{J+1}(k)}{(\log k)^{J+1}}\bigg)\bigg\},
\end{align}
where
$$
C^{*}_j(k)
:= \int_2^{2\sqrt{k}} \bigg(\log\frac{t}{2}\bigg)^j \frac{\log I_0(t)-t}{t^2} \d t.
$$
Since $\log I_0(t)= t + O(\log t)$ for $t\ge 2$, we have
\begin{align*}
	\int_{2\sqrt{k}}^{\infty} \bigg(\log\frac{t}{2}\bigg)^j \frac{\log I_0(t)-t}{t^2} \d t
	& \ll \int_{\sqrt{k}}^{\infty} \frac{(\log t)^{j+1}}{t^2} \d t
	\ll_j \frac{(\log k)^{j+1}}{\sqrt{k}},
\end{align*}
which implies that
$$
C^{*}_j(k)
= C^{*}_j + O_j\bigg(\frac{(\log k)^{j+1}}{\sqrt{k}}\bigg)
\quad\text{with}\quad
C^{*}_j
:= \int_2^{\infty} \bigg(\log\frac{t}{2}\bigg)^j \frac{\log I_0(t)-t}{t^2} \d t.
$$
Combining this with \eqref{3.5} and \eqref{3.6}, we obtain
$$
\sum_{\sqrt{k}<p\le k} \bigg(\log I_0\bigg(\frac{2k}{p}\bigg) - \frac{2k}{p}\bigg)
= \frac{2k}{\log k} \bigg\{\sum_{j=0}^{J} \frac{C^{*}_j}{(\log k)^j}
+ O_J\bigg(\frac{1}{(\log k)^{J+1}}\bigg)\bigg\}.
$$
Therefore we derive that
\begin{equation}\label{Pi_1}
	\Pi_1
	= \exp\bigg(\frac{2k}{\log k} \bigg\{\sum_{j=0}^{J} \frac{C^{*}_j}{(\log k)^j}
	+ O_J\bigg(\frac{1}{(\log k)^{J+1}}\bigg)\bigg\}\bigg).
\end{equation}

For $\Pi_2$, by the prime number theorem, we have similarly
\begin{align*}
	\sum_{k<p\le y}\log I_0\bigg(\frac{2k}{p}\bigg)
	& =\int_k^y\log I_0\left(\frac {2k}u\right)\frac{\d u}{\log u}
	+ O\big(k{\rm e}^{-c\sqrt{\log k}}\big)	\\
	& = \frac{2k}{\log k} \bigg\{\sum_{j=0}^J \frac{C^{**}_j}{(\log k)^j}
	+ O\bigg(\frac{k}{y}\bigg)\bigg\}
	+ O\big(k{\rm e}^{-c\sqrt{\log k}}\big),
\end{align*}
where
$$
C^{**}_j := \int_0^2 \bigg(\log\frac{t}{2}\bigg)^j \frac{\log I_0(t)}{t^2} \d t.
$$
Here we have used the inequality $\log I_0(t)\ll t^2$ for $0\le t< 2$ to evaluate the truncated integral that
\begin{align*}
	\int_0^{2k/y} \bigg(\log\frac{t}{2}\bigg)^j \frac{\log I_0(t)}{t^2} \d t
	\ll \int_0^{k/y} (\log t)^j \d t
	\ll_{A, j} \frac{k(\log k)^j}{y}\cdot
\end{align*}
Therefore we derive that
\begin{equation}\label{Pi_2}
	\Pi_2
	= \exp\bigg(\frac{2k}{\log k} \bigg\{\sum_{j=0}^J \frac{C^{**}_j}{(\log k)^j}
	+ O_{A, J}\bigg(\frac{k}{y}\bigg)\bigg\}\bigg).
\end{equation}
Now Proposition \ref{prop3.1} follows from \eqref{proof-th1:3}, \eqref{Pi_1} and \eqref{Pi_2}
with $C_j = C^{*}_j + C^{**}_j$.
\end{proof}

\vskip 6mm

\section{Moments of the short Euler products}

Let $\zeta(s; y)$ and $D_k(y)$ be defined as in \eqref{ded:zetasy} and \eqref{def:Dky}, respectively.
In this section, we shall evaluate the $2k$-th moment of $\zeta(1+{\rm i}t;y)$ by proving the following proposition.
This is essentially the first part of Theorem 3 in \cite{GS06}.
The main difference is a slightly enlarged length of the short Euler products,
which is important for the proof of Theorem \ref{thm1}.

\begin{proposition}\label{prop4.1}
Let $A>0$ be a constant. Then we have
	$$
	\frac1T\int_T^{2T} |\zeta(1+{\rm i}t; y)|^{2k} \d t
	= D_k(y)
	\bigg\{1+O_A\bigg(\exp\bigg\{-\frac{\log T}{2(\log_2T)^4}\bigg\}\bigg)\bigg\}
	$$
	uniformly for 
\begin{equation}\label{Cond:Tyk}
\begin{cases}
T\ge T_0(A),
\\\noalign{\vskip 0,5mm}
{\rm e}^2\log T\le y\le(\log T)(\log_2T)^{A},
\\\noalign{\vskip 0,5mm}
k\in \N\cap [2, (\log T)/({\rm e}^{10}\log(y/\log T))],
\end{cases}
\end{equation}
where the implied constant and the constant $T_0(A)$ depend on $A$ only.
\end{proposition}

We show that for $k$ and $y$ in \eqref{Cond:Tyk}, the diagonal terms lead to the main term, while the off-diagonal terms only contribute to the error term.
For this, we need to establish a preliminary lemma.

\subsection{A preliminary lemma}
If $2\le k\le 10^6$, we write $\mathcal{I}_0=(k, y]$, $\mathcal{I}_1=(1, k]$ and $J=0$.
When $k>10^6$, we take $J=\lfloor 4(\log_2k)/\log 2\rfloor$ and devide $(1,y]$ into $J+2$ intervals
\begin{align*}
(1,y] = \mathcal{I}_0\cup \mathcal{I}_1\cup \cdots\cup \mathcal{I}_{J+1},
\end{align*}
where $\mathcal{I}_0 := (k,y]$, $\mathcal{I}_j := (k/2^j, k/2^{j-1}]\;(1\le j\le J)$
and $\mathcal{I}_{J+1} :=  (1, k/2^J]\subset (1, 2k/(\log_2k)^4]$.
For each $j\in \{0, 1, \dots, J+1\}$, we use $S(\mathcal{I}_j)$ to represent the set of all positive integers which have prime divisors only in $\mathcal{I}_j$ ($1\in S(\mathcal{I}_j)$ by convention).
Recall that $S(y)$ is the set of $y$-friable integers.
Thus
\begin{equation}\label{n=n1cdotsnJ}
n\in S(y)
\;\Leftrightarrow\;
n \mathop{=}^{{\rm uniquely}} n_0\cdots n_{J+1}
\;\,\text{with}\;\,
n_j\in S(\mathcal{I}_j)
\;\, (0\le j\le J+1).
\end{equation}
Set 
\begin{equation}\label{def:Dkj}
D_{k, j} := \sum_{h\in S(\mathcal{I}_j)} d_k(h)^2/h^2
\end{equation}
such that
\begin{equation}\label{product:Dkj}
D_k(y) = D_{k, 0} D_{k, 1} \cdots D_{k, J+1}.
\end{equation}

We have the following lemma.

\begin{lemma}\label{lem4.2}
Let $G_0:=T^{1/5}$ and $G_j:=T^{1/(5j^2)} \; (j\ge 1)$. Then we have
\begin{equation}\label{eq:lem4.2}
\sum_{g\in S(\mathcal{I}_j), \, g>G_j} \frac{2^{\omega(g)}}{g}\sum_{h\in S(\mathcal{I}_j)}\frac {d_k(gh)d_k(h)}{h^2}
\le D_{k, j}\exp\bigg(\!-\frac{\log T}{(\log_2T)^4}\bigg)
\end{equation}
for $(T,y,k)$ in \eqref{Cond:Tyk} and $0\le j\le J+1$,
where $\omega(n)$ denotes the number of distinct prime factors of $n$.
\end{lemma}
	
\begin{proof}
This is essentially \cite[Lemma 5]{GS06}. The difference is that the upper bound of $y$ is shifted from $(\log T)(\log_2T)^{4}$ to $(\log T)(\log_2T)^{A}$ 
with arbitrary $A>0$. 
Note that this change is harmless to the range of $k$, 
since whether $y=(\log T)(\log_2T)^{4}$ or $y=(\log T)(\log_2T)^{A}$, 
it does not influence the upper bound of $k$: 
\begin{equation}\label{UB:k}
k\le \frac{\log T}{{\rm e}^{10}\log(y/\log T)}
\le \frac{\log T}{{\rm e}^{10}\log({\rm e}^2\log T/\log T)}
= \frac{\log T}{2{\rm e}^{10}}\cdot
\end{equation}
So this makes no difference so that we can follow Granville and Soundararajan's procedure.
Here we reproduce thier proof for convenience of the reader. 
Denote by $\mathfrak{S}_k(\mathcal{I}_j)$ the member on the left-hand side of \eqref{eq:lem4.2}.

Firstly we consider the case of $1\le j\le J+1$.
For $\delta=1/(2^{j/2}\log k)$,  by Rankin's trick and exchanging the order of the sums, we have
\begin{equation}\label{041}
\mathfrak{S}_k(\mathcal{I}_j)
\le \frac{1}{G_j^{\delta}} \sum_{g\in S(\mathcal{I}_j)}\frac{2^{\omega(g)}}{g^{1-\delta}}\sum_{h\in S(\mathcal{I}_j)}\frac {d_k(gh)d_k(h)}{h^2}
= \frac{1}{G_j^{\delta}} \sum_{h\in S(\mathcal{I}_j)}\frac {d_k(h)}{h^{1+\delta}}\sum_{g\in S(\mathcal{I}_j)}\frac{2^{\omega(g)}d_k(gh)}{(gh)^{1-\delta}}\cdot
\end{equation}
The inner sum is over part of $S(\mathcal{I}_j)$, so for any $h\in S(\mathcal{I}_j)$ we have 
$$
\sum_{g\in S(\mathcal{I}_j)}\frac{2^{\omega(g)}d_k(gh)}{(gh)^{1-\delta}}
\le \sum_{g\in S(\mathcal{I}_j)}\frac{2^{\omega(gh)}d_k(gh)}{(gh)^{1-\delta}}
\le \sum_{g\in S(\mathcal{I}_j)}\frac{2^{\omega(g)}d_k(g)}{g^{1-\delta}}\cdot
$$
Thus
\begin{equation}\label{042}
\begin{aligned}
\mathfrak{S}_k(\mathcal{I}_j)
& \le  \frac{1}{G_j^{\delta}}\sum_{h\in S(\mathcal{I}_j)}\frac {d_k(h)}{h^{1+\delta}}\sum_{g\in S(\mathcal{I}_j)}\frac{2^{\omega(g)}d_k(g)}{g^{1-\delta}}
\\
& =  \frac{1}{G_j^{\delta}} \prod_{p\in \mathcal{I}_j}
\bigg(1-\frac{1}{p^{1+\delta}}\bigg)^{-k}\bigg(2\bigg(1-\frac{1}{p^{1-\delta}}\bigg)^{-k}-1\bigg)
\\
& \le \frac{1}{G_j^{\delta}}\prod_{p\in \mathcal{I}_j}2\bigg(1-\frac1p\bigg)^{-2k}
\Xi_j(p)^{-k}
\end{aligned}
\end{equation}
with
$$
\Xi_j(p) 
:= \bigg(1-\frac{1}{p^{1+\delta}}\bigg)\bigg(1-\frac{1}{p^{1-\delta}}\bigg)\bigg(1-\frac1p\bigg)^{-2}
= 1-\frac{p(p^\delta+p^{-\delta}-2)}{(p-1)^2}\cdot
$$
Noticing that $p\in \mathcal{I}_j$ with $1\le j\le J+1$, we have $p\le k$.
Thus by the first inequality in \eqref{lem2.1:2} of Lemma \ref{lem2.1}, we can derive that
\begin{equation}\label{043} 
\prod_{p\in \mathcal{I}_j}\bigg(1-\frac1p\bigg)^{-2k}\le D_{k, j}\prod_{p\in \mathcal{I}_j}\frac{50k}{p},
\end{equation}
	while 
\begin{equation}\label{044}
\begin{aligned}
\Xi_j(p)
& = 1-\frac{2p}{(p-1)^2}\sum_{n=1}^{\infty}\frac{(\delta\log p)^{2n}}{(2n)!}	
\\
& \ge 1-\frac{2p(\delta\log p)^2}{(p-1)^2}\sum_{n=1}^{\infty}\frac{2^{-j(n-1)}}{(2n)!}	
\ge 1-\frac{c(2\delta\log p)^2}{p}\cdot
\end{aligned}
\end{equation}
with $c:=\sum_{n=1}^{\infty}\frac{2^{-(n-1)}}{(2n)!}<\frac{16}{31}$.
Combining (\ref{042})--(\ref{044}) and using the inequality $-\log(1-t)\le \sqrt{3}t \; (0\le t\le 2^{-1/2}c)$, we can derive that 
\begin{equation}\label{SkIj:UB-1}
\mathfrak{S}_k(\mathcal{I}_j) 
\le D_{k, j} 
\exp\bigg(\log\frac{1}{G_j^{\delta}}+\sum_{p\in \mathcal{I}_j}\bigg(\log\frac{100k}{p}+\frac{\sqrt{3}k(2\delta\log p)^2}{p}\bigg)\bigg).
\end{equation}
We choose $\delta=1/(2^{j/2}\log k)$. 
For $1\le j\le J+1$, we have 
\begin{align*}
\sum_{p\in \mathcal{I}_j}\bigg(\log\frac{100k}{p}+\frac{\sqrt{3}k(2\delta\log p)^2}{p}\bigg)
& \le \bigg(\log\bigg(\frac{100k}{k/2^j}\bigg)+\frac{\sqrt{3}k(2\delta\log (k/2^j))^2}{k/2^j}\bigg) 
\frac{2(k/2^{j-1})}{\log(k/2^{j-1})}
\\
& \le (j\log2+2\log 10+4\sqrt{3})\frac{4k}{2^j\log k}
\le \frac{50j k}{2^j\log k}
\\
& \le \frac{25j \log T}{2^j {\rm e}^{10} \log k}
\le \frac{25\log T}{2^{j/2} j^2 {\rm e}^6 \log k}
\end{align*}
thanks to \eqref{UB:k} and the inequality $j^3\le {\rm e}^{4} 2^{j/2}\;(j\ge 1)$.
Thus 
\begin{equation}\label{casei}
\begin{aligned}
\log\frac{1}{G_j^{\delta}} 
+ \sum_{p\in \mathcal{I}_j} \bigg(\log\frac{100k}{p}+\frac{\sqrt{3}k(2\delta\log p)^2}{p}\bigg)
& \le - \frac{1-125{\rm e}^{-6}}{5j^22^{j/2}}\cdot\frac{\log T}{\log k}
\\
& \le - \frac{1-125{\rm e}^{-6}}{5J^22^{J/2}}\cdot\frac{\log T}{\log k}
\\
& \le -\frac{\log T}{(\log k)^4}
\le -\frac{\log T}{(\log_2T)^4}
\end{aligned}
\end{equation}
for $(T,y,k)$ in \eqref{Cond:Tyk} and $1\le j\le J+1$.
Inserting (\ref{casei}) into \eqref{SkIj:UB-1}, then $1\le j\le J+1$ we have
\begin{align*}
\mathfrak{S}_k(\mathcal{I}_j)
& \le D_{k, j} \exp \bigg(\!-\frac{\log T}{(\log_2T)^4}\bigg).
	\end{align*}
	Thus the lemma in this case follows.

When $j=0$, by Rankin's trick and the trivial inequality $d_k(gh)\le d_k(g)d_k(h)$, 
we have for any $0<\delta<1$
\begin{equation}\label{case:j=0}
\mathfrak{S}_k(\mathcal{I}_0)
\le \frac{D_{k, 0}}{G_0^{\delta}} \sum_{g\in S(\mathcal{I}_0)}\frac{2^{\omega(g)}d_k(g)}{g^{1-\delta}}
= \frac{D_{k, 0}}{G_0^{\delta}} \prod_{p\in \mathcal{I}_0}\bigg(2\bigg(1-\frac{1}{p^{1-\delta}}\bigg)^{-k}-1\bigg).
\end{equation}
For $k<p\le y$, we have the upper bound
$$
2\bigg(1-\frac{1}{p^{1-\delta}}\bigg)^{-k}-1
\le 2\exp\bigg(\frac{2k}{p^{1-\delta}}\bigg)-1
\le \exp\bigg(\frac{3k}{p^{1-\delta}}\bigg)
\le \exp\bigg(\frac{3ky^{\delta}}{\log k}\cdot\frac{\log p}{p}\bigg).
$$
Inserting this into \eqref{case:j=0} and using \cite[Theorem I.1.7]{Te95} in form
$$
\sum_{k<p\le y} \frac{\log p}{p} \le \log\bigg(\frac{25y}{k}\bigg),
$$
we can derive that
\begin{align*}
\mathfrak{S}_k(\mathcal{I}_0)
\le \frac{D_{k, 0}}{G_0^{\delta}} \exp\bigg(\frac{3ky^{\delta}}{\log k}\log\bigg(\frac{25y}{k}\bigg)\bigg).
\end{align*}
Taking $\delta=1/(10\log_2T)$ and noticing that $t\mapsto (t/\log t)\log(25y/t)$ is increasing in $\mathcal{I}_0$,
we deduce, for $(T,y,k)$ in \eqref{Cond:Tyk},
\begin{align*}
\frac{ky^{\delta}}{\log k}\log\bigg(\frac{25y}{k}\bigg)
& \le \frac{{\rm e}^{1/5} (\log T)/({\rm e}^{10}\log(y/\log T))}{\log((\log T)/({\rm e}^{10}\log(y/\log T)))}
\log\bigg(\frac{25{\rm e}^{10}y}{\log T}\log\bigg(\frac{y}{\log T}\bigg)\bigg)
\\
& \le \frac{10{\rm e}^{1/5}\log T}{{\rm e}^{10}\log((\log T)/({\rm e}^{10}\log(y/\log T)))}
\\
& \le \frac{20{\rm e}^{1/5}\log T}{{\rm e}^{10}\log_2T}\cdot
\end{align*}
Inserting this into the preceding inequality, we have
\begin{align*}
\mathfrak{S}_k(\mathcal{I}_0)
\le D_{k, 0}\exp\bigg(\!-\delta\log G_0+\frac{20{\rm e}^{1/5}\log T}{{\rm e}^{10}\log_2T}\bigg)
\end{align*}
for $(T,y,k)$ in \eqref{Cond:Tyk}.
Now the result of Lemma \ref{lem4.2} follows by recalling that $G_0=T^{1/5}$.
	\end{proof}

\subsection{Proof of Proposition \ref{prop4.1}}

As in \cite{GS06}, we shall prove a more general result:
\textit{
Let $A>0$ be a constant and $\mathcal{R}\subset \{0, 1, \dots, J+1\}$. Then we have
\begin{equation}\label{R1}
	\frac1T\int_T^{2T} |\zeta(1+{\rm i}t; \mathcal{R})|^{2k} \d t
	= D_k(\mathcal{R})
	\bigg\{1+O_A\bigg(\exp\bigg\{\!-\frac{\log T}{2(\log_2T)^4}\bigg\}\bigg)\bigg\}
\end{equation}
	uniformly for $(T, y, k)$ in \eqref{Cond:Tyk},
where 
\begin{equation}\label{R2}
\mathcal{I}_{\mathcal{R}} := \mathop{\bigcup}_{r\in \mathcal{R}} \mathcal{I}_r,
\quad
\zeta(s; \mathcal{R})
:= \prod_{p\in \mathcal{I}_{\mathcal{R}}} (1-p^{-s})^{-1},
\quad
D_k(\mathcal{R})
:= \sum_{n\in S(\mathcal{I}_{\mathcal{R}})} d_k(n)^2/n^2
\end{equation}
and the implied constant depends on $A$ only.}

\vskip 1mm

We shall prove \eqref{R1} by induction on the cardinal of $\mathcal{R}$.
The case of $\mathcal{R}=\emptyset$ (i.e. $|\mathcal{R}\blue|=0$) is trivial, 
since $\zeta(s; \emptyset)=1=D_k(\emptyset)$.
Now we suppose that \eqref{R1} holds for all proper subset of $\mathcal{R}$ 
and prove that it is true for $\mathcal{R}$.

Firstly, in view of \eqref{n=n1cdotsnJ}, we have
\begin{equation}\label{4.3}
\begin{aligned}
\frac1T\int_T^{2T} |\zeta(1+{\rm i}t; \mathcal{R})|^{2k} \d t
& = \sum_{m,n\in S(\mathcal{I}_{\mathcal{R}})} \frac{d_k(m)d_k(n)}{mn}\frac1T\int_T^{2T} \bigg(\frac{n}{m}\bigg)^{{\rm i }t}\d t
\\
& = \sum_{\substack{m_r, n_r\in S(\mathcal{I}_r)\\ r\in \mathcal{R}}}
\bigg(\prod_{r\in \mathcal{R}} \frac{d_k(m_r)d_k(n_r)}{m_rn_r}\bigg)
\frac{1}{T} \int_T^{2T} \bigg(\prod_{r\in \mathcal{R}} \frac{n_r}{m_r}\bigg)^{{\rm i }t} \d t.
	\end{aligned}
	\end{equation}
Denote
$$
g_j=\frac{{\rm lcm}(m_j,n_j)}{{\rm gcd}(m_j,n_j)}
\qquad\text{and}\qquad 
h_j={\rm gcd}(m_j,n_j).
$$
Using the principle of inclusion-exclusion, we divide the sum in \eqref{4.3} into two parts
\begin{align}\label{4.4}
\sum_{\substack{m_r, n_r\in S(\mathcal{I}_r)\\ g_r\le G_r\\ \forall r\in \mathcal{R}}}
+ \sum_{\substack{\mathcal{W}\subset \mathcal{R}\\ \mathcal{W}\not=\emptyset}} (-1)^{|\mathcal{W}|}
\sum_{\substack{m_r, n_r\in S(\mathcal{I}_r), \, \forall \, r\in \mathcal{R}\\ g_w> G_w, \forall w\in \mathcal{W}}}
\end{align}
where the $G_j$ is defined as in Lemma \ref{lem4.2}.

In the first sum, the case $g_r=1 \; (r\in \mathcal{R})$ counts the diagonal terms 
and leads to the main term
$$
\sum_{n\in S(\mathcal{I}_{\mathcal{R}})} {d_k(n)^2}/{n^2} = D_k(\mathcal{R}).
$$
Otherwise, we have $\prod_{r\in \mathcal{R}} (m_r/n_r)\neq1$.
Therefore by $g_r\le G_r$ we have
$$
	\bigg|\log\prod_{r\in \mathcal{R}} \frac{m_r}{n_r}\bigg|
	= \bigg|\log\prod_{r\in \mathcal{R}} \frac{m_r/h_r}{n_r/h_r}\bigg|
	\ge\log\Big(1+\prod\limits_{r\in \mathcal{R}} g_r^{-1}\Big)
	\gg \prod\limits_{r\in \mathcal{R}} G_r^{-1}.
$$
Thus in these terms we have
\begin{align*}
	\bigg|\frac1T\int_T^{2T} \bigg(\prod_{r\in \mathcal{R}} \frac{m_r}{n_r}\bigg)^{{\rm i }t}\d t\bigg|
	\ll \frac1T\bigg|\log\prod\limits_{r\in \mathcal{R}} \frac{m_r}{n_r}\bigg|^{-1}
	\ll \frac{1}{T^{2/5}}\cdot
\end{align*}
Therefore the sum over these terms is 
\begin{align*}
& \ll \frac{1}{T^{2/5}} \sum_{m,n\in S(\mathcal{I}_{\mathcal{R}})}\frac{d_k(m)d_k(n)}{mn}
\ll \frac{1}{T^{2/5}} \prod_{p\in \mathcal{I}_{\mathcal{R}}}\bigg(1-\frac1p\bigg)^{-2k}.
\end{align*}
By (\ref{lem2.1:2}) in Lemma \ref{lem2.1} and the inequality $-\log(1-t)\le 2t\;(0\le t\le \frac{1}{2})$, we have
\begin{align*}
\prod_{p\in \mathcal{I}_{\mathcal{R}}} \bigg(1-\frac1p\bigg)^{-2k}
& \le \prod_{\substack{p\in \mathcal{I}_{\mathcal{R}}\\ p\le k}}
\bigg(\frac{50k}{p}\sum_{\nu\ge 0} \frac{d_k(p^{\nu})^2}{p^{2\nu}}\bigg)
\prod_{\substack{p\in \mathcal{I}_{\mathcal{R}}\\ k<p\le y}} \bigg(1-\frac1p\bigg)^{-2k}
\\
& \le D_k(\mathcal{R})
\exp\bigg(\sum_{p\le k}\log\frac{50k}{p}+\sum_{k<p\le y} \frac{4k}{p}\bigg)
\\
& \le D_k(\mathcal{R})
\exp\bigg(\frac{10k}{\log k}\log\frac{25y}{k}\bigg).
\end{align*}
Therefore, the contribution of the first sum in \eqref{4.4} is
\begin{equation}\label{Contribution:first-sum}
\sum_{\substack{m_r, n_r\in S(\mathcal{I}_r)\\ g_r\le G_r\\ \forall r\in \mathcal{R}}}
= D_k(\mathcal{R}) + O\bigg(\frac{D_k(\mathcal{R})}{T^{2/5}} \exp\bigg(\frac{12k}{\log k} \log\frac{25y}{k}\bigg)\bigg)
= D_k(\mathcal{R}) \bigg\{1+O\bigg(\frac{1}{T^{1/3}}\bigg)\bigg\}
\end{equation}
for $(T,y,k)$ as in \eqref{Cond:Tyk}.

Now consider the second sum in \eqref{4.4}.
For a given non-empty subset $\mathcal{W}\subset \mathcal{R}$, we have
$$
\sum_{\substack{m_r, n_r\in S(\mathcal{I}_r)\\ \forall \, r\in \mathcal{R}\sset \mathcal{W}}}
\bigg(\prod_{r\in \mathcal{R}\sset \mathcal{W}} \frac{d_k(m_r)d_k(n_r)}{m_rn_r}\bigg)
\bigg(\prod_{r\in \mathcal{R}\sset \mathcal{W}} \frac{n_r}{m_r}\bigg)^{{\rm i }t} 
= |\zeta(1+{\rm i}t; \mathcal{R}\sset \mathcal{W})|^{2k}.
$$
where $\zeta(s; \mathcal{R}\sset \mathcal{W})$ is defined as in \eqref{R2}.
Thus the inner sum in \eqref{4.4} gives
\begin{align*}
& \sum_{\substack{m_r, n_r\in S(\mathcal{I}_r), \, \forall \, r\in \mathcal{R}\\ g_w> G_w, \forall w\in \mathcal{W}}}
\bigg(\prod_{r\in \mathcal{R}} \frac{d_k(m_r)d_k(n_r)}{m_rn_r}\bigg)
\frac{1}{T} \int_T^{2T} \bigg(\prod_{r\in \mathcal{R}} \frac{n_r}{m_r}\bigg)^{{\rm i }t} \d t
\\
& = \sum_{\substack{m_w, n_w\in S(\mathcal{I}_w)\\ g_w> G_w\\ \forall w\in \mathcal{W}}}
\bigg(\prod_{w\in \mathcal{W}} \frac{d_k(m_w)d_k(n_w)}{m_wn_w}\bigg)
\frac{1}{T} \int_T^{2T} \bigg(\prod_{w\in \mathcal{W}} \frac{n_w}{m_w}\bigg)^{{\rm i }t} 
|\zeta(1+{\rm i}t; \mathcal{R}\sset \mathcal{W})|^{2k} \d t,
\end{align*}
which is bounded by
\begin{align}\label{(4.5)}
\sum_{\substack{m_w, n_w\in S(\mathcal{I}_w)\\ g_w> G_w\\ \forall w\in \mathcal{W}}}
\bigg(\prod_{w\in \mathcal{W}} \frac{d_k(m_w)d_k(n_w)}{m_wn_w}\bigg)
\frac{1}{T} \int_T^{2T} |\zeta(1+{\rm i}t; \mathcal{R}\sset \mathcal{W})|^{2k} \d t.
\end{align}
Observe that the integral does not depend on $w$, so we can change the order of sum and integral.
Further, we have
\begin{align}
\sum_{\substack{m_w, n_w\in S(\mathcal{I}_w)\\ g_w> G_w\\ \forall w\in \mathcal{W}}}
\bigg(\prod_{w\in \mathcal{W}} \frac{d_k(m_w)d_k(n_w)}{m_wn_w}\bigg)
= \prod_{w\in \mathcal{W}} 
\sum_{\substack{m_w, n_w\in S(\mathcal{I}_w)\\ g_w> G_w}}
\frac{d_k(m_w)d_k(n_w)}{m_wn_w}\cdot
\end{align}
For any multiplicative function $f$, we have $f(m)f(n)=f({\rm lcm}(m,n))f({\rm gcd}(m,n))$.
While the number of $(m,n)$ such that ${\rm gcd}(m,n)/{\rm lcm}(m,n)=g$, ${\rm gcd}(m,n)=h$ is $2^{\omega(g)}$.
Thus we derive that
\begin{equation}\label{(4.8)}
\begin{aligned}
\sum_{\substack{m_w, n_w\in S(\mathcal{I}_w)\\ g_w> G_w}}
\frac{d_k(m_w)d_k(n_w)}{m_wn_w}
& = \sum_{\substack{m_w, n_w\in S(\mathcal{I}_w)\\ g_w> G_w}}
\frac{d_k(g_wh_w)d_k(h_w)}{g_wh_w^2}
\\
& = \sum_{\substack{m_w, n_w\in S(\mathcal{I}_w)\\ g_w> G_w}}
\frac{2^{\omega(g_w)}d_k(g_wh_w)d_k(h_w)}{g_wh_w^2}\cdot
\end{aligned}
\end{equation}
Therefore by Lemma \ref{lem4.2}, this is bounded by $D_k(\mathcal{W})\exp(-(\log T)/(\log_2T)^4)$.

Summarizing \eqref{(4.5)}-\eqref{(4.8)}, we deduce that the second sum in \eqref{4.4} is bounded by
\begin{align}\label{(4.9)}
\sum_{\substack{\mathcal{W}\subset \mathcal{R}\\ \mathcal{W}\not=\emptyset}} 
D_k(\mathcal{W})
\bigg(\frac{1}{T} \int_T^{2T}|\zeta(1+{\rm i}t; \mathcal{R}\sset \mathcal{W})|^{2k}\d t\bigg)
\exp\bigg(\!-\frac{\log T}{(\log_2T)^4}\bigg).
\end{align}
According to the induction hypothesis, it follows that
\begin{align}\label{(4.10)}
\frac{1}{T} \int_T^{2T} |\zeta(1+{\rm i}t; \mathcal{R}\sset \mathcal{W})|^{2k}\d t
\ll D_k(\mathcal{R}\sset \mathcal{W}).
\end{align}
Then noticing that $D_k(\mathcal{W})D_k(\mathcal{R}\sset \mathcal{W})=D_k(\mathcal{R})$,
\eqref{(4.9)} is bounded by
\begin{equation}\label{R3}
2^{L+1} D_k(\mathcal{R}) \exp\bigg(\!-\frac{\log T}{(\log_2T)^4}\bigg)
\ll D_k(\mathcal{R}) \exp\bigg(\!-\frac{\log T}{2(\log_2T)^4}\bigg)
\end{equation}
since $2^{L+1}\ll \exp(2\log_2k)\ll \exp(2\log_3T)$.
Now the desired result follows from \eqref{Contribution:first-sum} and \eqref{R3}.
\hfill
$\square$

\vskip 6mm

\section{Proof of Theorem \ref{thm1}}

Firstly we recall the definition of of the short Euler products
$$
\zeta(s; y) := \prod_{p\le y} \Big(1-\frac{1}{p^s}\Big)^{-1}
$$
and the definition of its distribution function:
$$
\Phi_T(\tau; y) := \frac{1}{T}  {\rm meas}\big\{t\in[T,2T]:|\zeta(1+{\rm i}t; y)|>{\rm e}^{\gamma}\tau\big\}.
$$
The aim of this section is to prove the following result.

\begin{proposition}\label{prop5.1}
	Let $A>0$ be any constant and let $J\ge 1$ be an integer, and $\varepsilon$ satisfying \eqref{UB:varepsilon}. Then we have
\begin{equation}\label{(5.2)}
	\Phi_T(\tau+\varepsilon;y)
	\le \exp\bigg(\!-\frac{2{\rm e}^{\tau-C_0-1}}{\tau}
	\bigg\{1+\sum_{j=1}^J \frac{\mathfrak{a}_j}{\tau^j}
	+ O_J\bigg(\frac{1}{\tau^{J+1}} + \frac{{\rm e}^{\tau}}{y}\bigg)\bigg\}\bigg)
	\le\Phi_T(\tau-\varepsilon;y)
\end{equation}
	uniformly for
	\begin{equation}\label{(5.3)}
		T\ge T_0(A),
		\quad
		{\rm e}^2\log T\le y\le(\log T)(\log_2T)^{A},
		\quad
		2\le \tau\le \log_2T-20,
	\end{equation}
	where the $\mathfrak{a}_j$ and $C_0$ are the same as in Theorem \ref{thm1},
$T_0(A)$ is a positive constant depending on $A$ and the implied constant depends on $A$ and $J$ at most.
\end{proposition}

\subsection{Two preliminary lemmas}

In the following lemma, we will see the correlation between the distribution function and the moments of the short Euler products:

\begin{lemma}\label{lem5.2}
	Let $A>0$ be any constant and let $J\ge 1$ be an integer. Then we have
	\begin{equation}\label{(4.1)}
		\int_0^{\infty} \Phi_T(t; y) t^{2\kappa-1} \d t
		= \frac{(\log\kappa)^{2\kappa}}{2\kappa}
		\exp\bigg(\frac{2\kappa}{\log\kappa}
		\bigg\{\sum_{j=0}^J \frac{C_j}{(\log\kappa)^j} + O_{A, J}\bigg(\frac{\kappa}{y}+\frac{1}{(\log\kappa)^{J+1}}\bigg)\bigg\}\bigg)
	\end{equation}
	uniformly for
	\begin{equation}\label{(4.2)}
		T\ge 2,
		\quad
		{\rm e}^2\log T\le y\le(\log T)(\log_2T)^{A},
		\quad
		2\le \kappa\le (\log T)/({\rm e}^{10}\log(y/\log T)),
	\end{equation}
	where the $C_j$ are defined as in \eqref{def:Cj} and the implied constant depends on $A$ and $J$ at most.
\end{lemma}

\begin{proof}
	For any $\kappa>0$, we have
	\begin{align*}
		\int_0^{\infty} \Phi_T(u; y) u^{2\kappa-1} \d u
		& = \frac{1}{T} \int_0^{\infty}
		\bigg(\mathop{\int_T^{2T}}_{{\rm e}^{-\gamma} |\zeta(1+{\rm i}t; y)|>u} 1 \d t\bigg) u^{2\kappa-1} \d u
		\\
		& = \frac{1}{T} \int_T^{2T}
		\bigg(\int_0^{{\rm e}^{-\gamma} |\zeta(1+{\rm i}t; y)|} u^{2\kappa-1} \d u\bigg) \d t
		\\\noalign{\vskip 1mm}
		& = \frac{1}{T} \int_T^{2T} \frac{1}{2\kappa} ({\rm e}^{-\gamma} |\zeta(1+{\rm i}t; y)|)^{2\kappa} \d t,
	\end{align*}
	i.e.
	\begin{equation}\label{(4.3)}
		2\kappa\int_0^{\infty}\Phi_T(t; y)t^{2\kappa-1} \d t
		= \frac{{\rm e}^{-2\kappa\gamma}}{T} \int_T^{2T} |\zeta(1+{\rm i}t; y)|^{2\kappa} \d t.
	\end{equation}
	Now \eqref{(4.1)} follows from Propositions \ref{prop3.1} and \ref{prop4.1}
	when $\kappa$ is an integer.
	
	Next let $\kappa\notin \N$ be a real number verifying \eqref{(4.2)}.
	There is a unique integer $k$ verifying \eqref{(4.2)} such that $k-1<\kappa<k$.
	The formula \eqref{(4.3)} with $\kappa=\frac{1}{2}$ and \cite[Theorem 3]{GS06} imply that
	$$
	\int_0^{\infty}\Phi_T(u; y) \d u
	= \frac{{\rm e}^{-\gamma}}{T} \int_T^{2T} |\zeta(1+{\rm i}t; y)| \d t
	\le {\rm e}^{-\gamma} \bigg(\frac{1}{T} \int_T^{2T} |\zeta(1+{\rm i}t; y)|^4 \d t\bigg)^{1/4}
	\ll 1.
	$$
	Now for any $b>a>0$, by the H\"older inequality, it follows that
	$$
	\int_0^{\infty}\Phi_T(t; y) t^a \d t
	\le \bigg(\int_0^{\infty}\Phi_T(t; y) \d t\bigg)^{1-a/b} \bigg(\int_0^{\infty}\Phi_T(t; y) t^b \d t\bigg)^{a/b}.
	$$
	Thus there are two absolute positive constants $C$ and $D$ such that
	\begin{align*}
		\int_0^{\infty}\Phi_T(t; y) t^a \d t
		& \le C \bigg(\int_0^{\infty}\Phi_T(t; y) t^b \d t\bigg)^{a/b},
		\\
		\int_0^{\infty}\Phi_T(t; y) t^b \d t
		& \ge \bigg(D\int_0^{\infty}\Phi_T(t; y) t^a \d t\bigg)^{b/a}.
	\end{align*}
	Applying the first inequality with $(a, b)=(2\kappa-1, 2k-1)$
	and the second inequality with $(a, b)=(2k-3, 2\kappa13)$ respectively,
	we can obtain that
	$$
	\bigg(D\int_0^{\infty}\Phi_T(t; y) t^{2(k-1)-1} \d t\bigg)^{\frac{2\kappa-1}{2k-3}}
	\le \int_0^{\infty}\Phi_T(t; y) t^{2\kappa-1} \d t
	\le C \bigg(\int_0^{\infty}\Phi_T(t; y) t^{2k-1} \d t\bigg)^{\frac{2\kappa-1}{2k-1}}.
	$$
	On the other hand, seting
	$f(u) := \frac{2u}{\log u}\sum_{j=0}^J \frac{C_j}{(\log u)^j}$,
	then
	$f'(u) = -\frac{2}{\log u} \sum_{j=0}^J \frac{jC_j}{(\log u)^j}$.
	Thus
	\begin{equation}\label{kappa-1.k.kappa}
		f(k-1) = f(\kappa) + O(1)
		\qquad\text{and}\qquad
		f(k) = f(\kappa) + O(1).
	\end{equation}
	Now we can obtain \eqref{(4.1)} for $\kappa\notin \N$
	by substituting \eqref{(4.1)} for integers $k-1$ and $k$ and by using \eqref{kappa-1.k.kappa}.
	This completes the proof of Lemma \ref{lem5.2}.
\end{proof}

\begin{lemma}\label{lem5.3}
	Let $\{a_j\}_{j\ge 0}$ be a sequence of real numbers
	and $J\ge 0$ be an integer.
	If
	\begin{equation}\label{lem4.2:1}
		\tau = \log\kappa + a_0 + \sum_{j=1}^J \frac{a_j}{(\log\kappa)^j} + O_J\bigg(\frac{1}{(\log\kappa)^{J+1}}\bigg)
		\qquad
		(k\to\infty),
	\end{equation}
	then there is a sequence of real numbers $\{b_j\}_{j\ge 0}$ such that
	\begin{equation}\label{lem4.2:2}
		\log\kappa = \tau + b_0 + \sum_{j=1}^J \frac{b_j}{\tau^j} + O_J\bigg(\frac{1}{\tau^{J+1}}\bigg)
		\qquad
		(\tau\to\infty).
	\end{equation}
	Further we have $b_0=-a_0$ and $b_1=-a_1$ .
\end{lemma}

\begin{proof}
	We shall reason by recurrence on $J$.
	Taking $J=0$ in \eqref{lem4.2:1}, we have
	$$
	\tau = \log\kappa + a_0 + O\Big(\frac{1}{\log\kappa}\Big)
	\qquad
	(k\to\infty).
	$$
	From this we easily deduce that
	$$
	\log\kappa = \tau - a_0 + O\Big(\frac{1}{\tau}\Big)
	\qquad
	(\tau\to\infty).
	$$
	This is \eqref{lem4.2:2} with $J=0$ and $b_0=-a_0$.
	Suppose that
	\begin{equation}\label{lem4.2:3}
		\tau = \log\kappa + \sum_{j=0}^{J+1} \frac{a_j}{(\log\kappa)^j} + O_J\bigg(\frac{1}{(\log\kappa)^{J+3}}\bigg)
		\qquad
		(k\to\infty).
	\end{equation}
	Clearly this implies \eqref{lem4.2:1}.
	Thus according to the hypothesis of recurrence, \eqref{lem4.2:2} holds.
	Using \eqref{lem4.2:2} and \eqref{lem4.2:3}; we can derive that
	\begin{align*}
		\log\kappa
		& = \tau - a_0
		+ \sum_{j=1}^{J+1} \frac{a_j}{(\log\kappa)^j} + O_J\bigg(\frac{1}{\tau^{J+2}}\bigg)
		\\
		& = \tau - a_0
		+ \sum_{j=1}^{J+1} \frac{a_j}{\tau^j}
		\bigg\{1 + \sum_{d=1}^{J+2-j} \frac{b_{d-1}}{\tau^d} + O_J\Big(\frac{1}{\tau^{J+2-j}}\Big)\bigg\}^{-j}
		+ O_J\bigg(\frac{1}{\tau^{J+2}}\bigg).
	\end{align*}
	This implies the required result via the Taylor development of $(1-t)^{-j}$.
\end{proof}

\subsection{Proof of Proposition \ref{prop5.1}}
Let $\varepsilon\in [c(\log\kappa)^{-J-1}, \, 9c(\log\kappa)^{-J-1}]$ be a parameter to be chosen later,
where $c$ is a large constant.
Without loss of generality, we can suppose
\begin{equation}\label{UB:varepsilon}
	\varepsilon\le (\log\kappa)^{-J},
	\qquad
	\varepsilon^2\le (\log\kappa)^{-J-1}
\end{equation}
for $k\ge k_0$,
where $\kappa_0=\kappa_0(c)$ is a constant depending $c$.
Put $K=\kappa {\rm e}^{\varepsilon}$.
Noticing that $(\frac{t}{\tau+\varepsilon})^{2K-2\kappa}\ge 1$ for $t\ge \tau+\varepsilon$, we have
$$
2\kappa\int_{\tau+\varepsilon}^{\infty}\Phi_T(t; y)t^{2\kappa-1} \d t
\le (\tau+\varepsilon)^{2\kappa -2K}\bigg(2K\int_0^{\infty}\Phi_T(t; y)t^{2K-1} \d t\bigg).
$$
From this and Lemma \ref{lem5.2}, we deduce that
\begin{equation}\label{4.5}
	\frac{\int_{\tau+\varepsilon}^{\infty}\Phi_T(t; y)t^{2\kappa-1} \d t}
	{\int_0^{\infty}\Phi_T(t; y)t^{2\kappa-1} \d t}
	\le \exp\bigg\{2(g(K, \tau)-g(\kappa, \tau)) + O_J\bigg(\frac{\kappa^2}{y} + \frac{\kappa}{(\log\kappa)^{J+3}}\bigg)\bigg\}
\end{equation}
uniformly for $(T, y, \kappa)$ in \eqref{(4.2)} above, where
\begin{equation}\label{def:gktau}
	g(\kappa, \tau)
	:= - \kappa\log\bigg(\frac{\tau+\varepsilon}{\log\kappa}\bigg)
	+ \frac{\kappa}{\log\kappa} \sum_{j=0}^{J+1} \frac{C_j}{(\log\kappa)^j}\cdot
\end{equation}
Let $\tau_0=\tau_0(c, J)$ be a suitable constant depending on $c$ and $J$.
For $\tau_0\le \tau\le {\log_2T-20}$, take $\kappa=\kappa_{\tau}$ such that
\begin{equation}\label{def:k-tau}
	\tau = \log\kappa + a_0 + \sum_{j=1}^{J+1} \frac{a_j}{(\log\kappa)^j},
\end{equation}
where the $a_j = a_j(C_0, \dots, C_j)$ are constants to be determined later.
Our choice of $\tau_0$ guarantees that $\tau\ge \tau_0\Rightarrow \kappa\ge \kappa_0$,
which guarantees that
\begin{equation}\label{bound:varepsilon}
	\varepsilon\le 9c (\log\kappa)^{-J-1}\le (\log\kappa)^{-J}
	\quad\text{and}\quad
	\varepsilon^2\le 81c^2 (\log\kappa)^{-2J-2}\le (\log\kappa)^{-2J-1}.
\end{equation}
These bounds will be used often and all implied constants in the $O$-symbol is independent of $c$.
In view of \eqref{def:k-tau} and the Taylor formula, we can write
\begin{align*}
	g(\kappa, \tau)
	& = - \kappa\log\bigg(1 + \frac{a_0+\varepsilon}{\log\kappa}
	+ \frac{1}{\log\kappa}\sum_{j=1}^{J+1} \frac{a_j}{(\log\kappa)^j}\bigg)
	+ \frac{\kappa}{\log\kappa} \sum_{j=0}^{J+1} \frac{C_j}{(\log\kappa)^j}
	\\
	& = - \kappa \bigg(\frac{a_0-C_0+\varepsilon}{\log\kappa}
	+ \frac{a_1-a_0^2-C_1+a_0\varepsilon}{(\log\kappa)^2}
	+ \sum_{j=2}^{J+1} \frac{a_j-a^{*}_j-C_j}{(\log\kappa)^{j+1}}
	\bigg)
	+ O_J\bigg(\frac{1}{(\log\kappa)^{J+3}}\bigg),
\end{align*}
where the $a^{*}_j=a^{*}_j(a_0, \dots, a_{j-1})$ are constants
depending on $a_0, \dots, a_{j-1}$.
Take
$$
a_0 = C_0+1,
\quad
a_1 = C_0^2+C_0+C_1+2,
\quad
a_j = a^{*}_j(a_0, \dots, a_{j-1})+C_j
\;\;
(2\le j\le J+1).
$$
Thus
$$
g(\kappa, \tau)
= - \kappa \bigg(\frac{1+\varepsilon}{\log\kappa} + \frac{a_1-a_0^2-C_1+a_0\varepsilon}{(\log\kappa)^2}\bigg)
+ O_J\bigg(\frac{\kappa}{(\log\kappa)^{J+3}}\bigg).
$$
Let
$$
\mathfrak{T}
:= \log K + a_0 + \sum_{j=1}^{J+1} \frac{a_j}{(\log K)^j},
$$
then
$$
g(K, \mathfrak{T})
= - K \bigg(\frac{1+\varepsilon}{\log K} + \frac{a_1-a_0^2-C_1+a_0\varepsilon}{(\log K)^2}\bigg)
+ O_J\bigg(\frac{\kappa}{(\log\kappa)^{J+3}}\bigg).
$$
From these, we easily deduce that
\begin{align*}
	g(K, \mathfrak{T}) - g(\kappa, \tau)
	& = - K \bigg(\frac{1+\varepsilon}{\log K}-\frac{1+\varepsilon}{\log\kappa}
	+ \frac{a_1-a_0^2-C_1+a_0\varepsilon}{(\log K)^2}-\frac{a_1-a_0^2-C_1+a_0\varepsilon}{(\log\kappa)^2}\bigg)
	\\
	& \quad
	- (K-\kappa) \bigg(\frac{1+\varepsilon}{\log\kappa} + \frac{a_1-a_0^2-C_1+a_0\varepsilon}{(\log\kappa)^2}\bigg)
	+ O_J\bigg(\frac{\kappa}{(\log\kappa)^{J+3}}\bigg).
\end{align*}
Using \eqref{bound:varepsilon}, a simple computation shows that

\begin{equation}\label{4.10B}
	\begin{aligned}
		g(K, \mathfrak{T}) - g(\kappa, \tau)
		& = K \bigg(\frac{(1+\varepsilon)\varepsilon}{(\log K)\log\kappa}
		+ \frac{(a_1-a_0^2-C_1+a_0\varepsilon)\varepsilon \log(K\kappa)}{(\log K)^2(\log\kappa)^2}\bigg)
		\\
		& \quad
		- (K-\kappa) \bigg(\frac{1+\varepsilon}{\log\kappa} + \frac{a_1-a_0^2-C_1+a_0\varepsilon}{(\log\kappa)^2}\bigg)
		+ O_J\bigg(\frac{\kappa}{(\log\kappa)^{J+3}}\bigg)
		\\\noalign{\vskip 0,5mm}
		& = \frac{{\rm e}^{\varepsilon}\varepsilon\kappa}{(\log\kappa)^2}
		- ({\rm e}^{\varepsilon}-1) \kappa \bigg(\frac{1+\varepsilon}{\log\kappa} + \frac{a_1-a_0^2-C_1}{(\log\kappa)^2}\bigg)
		+ O_J\bigg(\frac{\kappa}{(\log\kappa)^{J+3}}\bigg).
	\end{aligned}
\end{equation}
On the other hand, in view of \eqref{def:gktau}, we have
$$
\frac{\partial g}{\partial \tau}(\kappa, \tau)
= - \frac{\kappa}{\tau+\varepsilon}\cdot
$$
Thus we have, for some $\eta_{\kappa}\in (\tau, \mathfrak{T})$,
\begin{equation}\label{4.11B}
	\begin{aligned}
		g(K, \tau) - g(K, \mathfrak{T})
		& = \frac{\partial g}{\partial \tau}(K, \eta_{\kappa}) (\tau - \mathfrak{T})
		\le \frac{\varepsilon K}{\tau+\varepsilon} \bigg\{1+O_J\bigg(\frac{1}{(\log\kappa)^2}\bigg)\bigg\}
		\\
		& = \frac{\varepsilon {\rm e}^{\varepsilon}\kappa}{\log\kappa}
		\bigg(1-\frac{a_0}{\log\kappa}\bigg)
		+ O_J\bigg(\frac{\kappa}{(\log\kappa)^{J+3}}\bigg).
	\end{aligned}
\end{equation}
Writing
$$
g(K, \tau) - g(\kappa, \tau)
= g(K, \mathfrak{T}) - g(\kappa, \tau) + g(K, \tau) - g(K, \mathfrak{T})
$$
and using \eqref{4.10B} and \eqref{4.11B}, we can derive that
\begin{align*}
	& g(K, \tau) - g(k, \tau)
	\\
	& \le - \big((a_1-a_0^2-C_1)({\rm e}^{\varepsilon}-1)
	+ C_0 \varepsilon {\rm e}^{\varepsilon}\big) \frac{\kappa}{(\log\kappa)^2}
	- ({\rm e}^{\varepsilon}-1-\varepsilon) \frac{\kappa}{\log\kappa}
	+ O_J\bigg(\frac{\kappa}{(\log\kappa)^{J+3}}\bigg)
	\\
	& \le - (a_1-a_0^2-C_1 + C_0) \frac{\varepsilon \kappa}{(\log\kappa)^2}
	+ O_J\bigg(\frac{\kappa}{(\log\kappa)^{J+3}}\bigg)
	\\
	& = - \frac{\varepsilon \kappa}{(\log\kappa)^2}
	+ O_J\bigg(\frac{\kappa}{(\log\kappa)^{J+3}}\bigg),
\end{align*}
thanks to the choice of $a_1=a_0^2+C_1 - C_0 + 1=C_0^2+C_0+C_1+2$.
Thus the inequality \eqref{4.5} can be written as
$$
\frac{\int_{\tau+\varepsilon}^{\infty}\Phi_T(t; y)t^{2\kappa-1} \d t}
{\int_0^{\infty}\Phi_T(t; y)t^{2\kappa-1} \d t}
\le \exp\bigg\{-\frac{2\varepsilon \kappa}{(\log\kappa)^2}
+ O_J\bigg(\frac{\kappa^2}{y} + \frac{\kappa}{(\log\kappa)^{J+3}}\bigg)\bigg\}
$$
for $\tau_0\le \tau\le \log_2T-20$ and $\kappa=\kappa_\tau$.
This implies that
\begin{equation}\label{4.7}
	\int_{\tau+\varepsilon}^{\infty}\Phi_T(t; y)t^{2\kappa-1} \d t
	\le \frac{1}{4} \int_0^{\infty}\Phi_T(t; y)t^{2\kappa-1} \d t,
\end{equation}
provided the constant $c$ is convenably large and $y\ge \kappa(\log\kappa)^{J+3}$.
Similarly
\begin{equation}\label{4.8}
	\int_0^{\tau-\varepsilon} \Phi_T(t; y)t^{2\kappa-1} \d t
	\le \frac{1}{4} \int_0^{\infty}\Phi_T(t; y)t^{2\kappa-1} \d t.
\end{equation}
From \eqref{4.7} and \eqref{4.8}, we deduce that
$$
\frac{1}{2}\int_0^{\infty}\Phi_T(t; y)t^{2\kappa-1} \d t
\le \int_{\tau-\varepsilon}^{\tau+\varepsilon} \Phi_T(t; y)t^{2\kappa-1} \d t
\le \int_0^{\infty}\Phi_T(t; y)t^{2\kappa-1} \d t.
$$
Combining this with Lemma \ref{lem5.2} leads to
\begin{equation}\label{4.9}
\begin{aligned}
& \int_{\tau-\varepsilon}^{\tau+\varepsilon} \Phi_T(t; y)t^{2\kappa-1} \d t
\\
& = \frac{(\log\kappa)^{2\kappa}}{2\kappa}
	\exp\bigg(\frac{2\kappa}{\log\kappa}
	\bigg\{\sum_{j=0}^J \frac{C_j}{(\log\kappa)^j} + O_J\bigg(\frac{\kappa}{y}+\frac1{(\log \kappa)^{J+1}}\bigg)\bigg\}\bigg)
\end{aligned}
\end{equation}
uniformly for $(T, y, \kappa)$ in \eqref{(4.1)} above and \eqref{(4.2)}.

On the other hand, in view of the fact that $\Phi_T(t; y)$ is decreasing in $t$, we have
\begin{equation}\label{4.10}
	\Phi_T(\tau+\varepsilon; y) (\tau-\varepsilon)^{2\kappa-1}
	\le 2k\int_{\tau-\varepsilon}^{\tau+\varepsilon} \Phi_T(t; y)t^{2\kappa-1} \d t
	\le \Phi_T(\tau-\varepsilon; y) (\tau+\varepsilon)^{2\kappa-1}
\end{equation}
Since $\tau=\log\kappa+\sum_{j=0}^{J+1} a_j/(\log\kappa)^j$ and $\varepsilon\asymp (\log\kappa)^{-J-1}$,
it follows that
\begin{equation}\label{4.11}
	\begin{aligned}
		\bigg(\frac{\tau\pm\varepsilon}{\log\kappa}\bigg)^{-2k} \frac{\tau\pm\varepsilon}{2\kappa}
		& = \exp\bigg(\!\! -2\kappa\log\bigg\{1+\sum_{j=0}^J \frac{a_j}{(\log\kappa)^{j+1}}
		+ O\bigg(\frac{1}{(\log\kappa)^{J+3}}\bigg)\bigg\}\bigg)
		\\
		& = \exp\bigg(\!\! -\frac{2\kappa}{\log\kappa}\bigg\{a_0+\sum_{j=1}^J \frac{\widetilde{a}_j}{(\log\kappa)^j}
		+ O\bigg(\frac{1}{(\log\kappa)^{J+1}}\bigg)\bigg\}\bigg).
	\end{aligned}
\end{equation}
where the $\widetilde{a}_j$ are constants ($\widetilde{a}_1=a_1+a_0^2$).
From \eqref{4.9}--\eqref{4.11}, we can deduce that
$$
\Phi_T(\tau+\varepsilon;y)
\le \exp\bigg(-\frac{2\kappa}{\log\kappa}\bigg\{2+\sum_{j=1}^J \frac{\widetilde{a}_j-C_j}{(\log\kappa)^j}
+ O_J\bigg(\frac{\kappa}{y}+\frac{1}{(\log\kappa)^{J+1}}\bigg)\bigg\}\bigg)
\le\Phi_T(\tau-\varepsilon;y).
$$
Since $\tau=\log\kappa+\sum_{j=0}^J a_j/(\log\kappa)^j$,
we can apply Lemma \ref{lem4.2} to write
$$
\frac{2\kappa}{\log\kappa}
= \frac{2{\rm e}^{\tau-C_0-1 + \sum_{j=1}^J b_j/\tau^j}}{\tau+\sum_{j=0}^J b_j/\tau^j}
= \frac{2{\rm e}^{\tau-C_0-1}}{\tau}
\bigg\{1+\sum_{j=1}^J \frac{b_j'}{\tau^j} + O\bigg(\frac{1}{\tau^{J+1}}\bigg)\bigg\}
$$
and
$$
\sum_{j=1}^J \frac{\widetilde{a}_j-C_j}{(\log\kappa)^j}
= \sum_{j=1}^J \frac{\widetilde{a}_j-C_j}{(\tau+\sum_{\ell=0}^J b_{\ell}/\tau^{\ell})^j}
= \sum_{j=1}^J \frac{c_j}{\tau^j} + O\bigg(\frac{1}{\tau^{J+1}}\bigg).
$$
Combining \eqref{4.9}--\eqref{4.11}, we obtain
\begin{equation}\label{4.12}
	\Phi_T(\tau+\varepsilon;y)
	\le \exp\bigg(\!-\frac{2{\rm e}^{\tau-C_0-1}}{\tau}
	\bigg\{1+\sum_{j=1}^J \frac{\mathfrak{a}_j}{\tau^j}
	+ O_J\bigg(\frac{1}{\tau^{J+1}} + \frac{{\rm e}^{\tau}}{y}\bigg)\bigg\}\bigg)
	\le\Phi_T(\tau-\varepsilon;y)
\end{equation}
with
\begin{align*}
	\mathfrak{a}_1
	& = 2b_1'+c_1
	= 2(b_1-b_0)+\widetilde{a}_1-C_1
	= 2(-a_1+a_0)+a_1+a_0^2-C_1
	\\
	& = 2a_0 - a_1 + a_0^2 - C_1
	= 2a_0 - (a_0^2+C_1) + a_0^2 - C_1
	= 2a_0 - 2C_1
	= 2(1+C_0 - C_1).
\end{align*}

\subsection{End of the proof of Theorem \ref{thm1}}
By Lemma \ref{lem3}, we can derive that
\begin{equation}\label{4.13}
	\Phi_T(\tau)
	= \Phi_T(\tau+O(\varepsilon+\eta); y)
	+ O(\exp(-(\log T)/(50\log_2T))
\end{equation}
with $\eta:=\sqrt{(\log T)/y}$.
Combining \eqref{4.12} and \eqref{4.13}, we can obtain
\begin{align*}
& \Phi_T(\tau)
\\
& = \exp\bigg(\!-\frac{2{\rm e}^{\tau-C_0-1}}{\tau}
	\bigg\{1+\sum_{j=1}^J \frac{\mathfrak{a}_j}{\tau^j}
	+ O_J\bigg(\frac{1}{\tau^{J+1}} + \sqrt{\frac{\log T}{y}}\bigg)\bigg\}\bigg)
		+ O\bigg(\!\exp\bigg(\!-\frac{\log T}{50\log_2T}\bigg)\bigg).
\end{align*}
This implies the required result by choosing
$y = \min\{{(\log T) \tau^{2J+2},(\log T)^2/{\rm e}^{10+\tau}}\}$.
\hfill
$\square$

\vskip 5mm

\textbf{Acknowledgement}.
The author would like to thank professor Jie Wu, for his suggestion on exploring this subject and his generous help in overcoming some difficulties, and Leo Goldmakher for helpful remarks regarding the paper of Granville and Soundararajan \cite{GS06}. The author is also grateful to Bin Wei for carefully reading the manuscript and valuable advice, and Jinjiang Li for helping correct some typos. The author is supported by the China Scholarship Council (CSC) for his study in France.


\vskip 5mm

\begin{thebibliography}{99}

\bibitem{Ai19}
C. Aistleitner, K. Mahatab and M. Munsch, 
{\it Extreme values of the Riemann zeta function on the $1$-line}, 
Int. Math. Res. Not. IMRN \textbf{22} (2019), 6924--6932.
	
\bibitem{BG13}
J. Bober and L. Goldmakher,
{\it The distribution of the maximum of character sums},
Mathematika. \textbf{59} (2013), 427--442.
	
\bibitem{BGGK}
	J. Bober, L. Goldmakher, A. Granville and D. Koukoulopoulos, {\it The frequency and the structure of large character sums}, J. Eur. Math. Soc. \textbf{20} (2018), 1759--1818.
	
	\bibitem{BohrJessen}
	H. Bohr and B. Jessen, {\it \"Uber die Werteverteilung der Riemannnschen Zetafunktion}, Acta Math. \textbf{54} (1930), 1--35; Zweite Mitteilung, ibid. \textbf{58} (1932), 1--55.
 
\bibitem{GS06}
	A. Granville and K. Soundararajan,
	{\it Extreme values of $|\zeta(1+{\rm i}t)|$},
	In: The Riemann Zeta Functionn and Related Themes:
	Papers in Honour of Professor K. Ramachandra, Ramanujan Math. Soc. Lect. Notes Ser. 2 (2006),
	Mysore, 65--80.
	
\bibitem{GS03}
	A. Granville and K. Soundararajan,
	{\it The distribution of values of $L(1,\chi_d)$},
	Geom. Funct. Anal. \textbf{13} (2003), 992--1028.
	
\bibitem{La2008}
	Y. Lamzouri,
	{\it The two-dimensional distribution of values of $\zeta(1+{\rm i}t)$},
	Int. Math. Res. Not. IMRN (2008), Art. ID rnn 106, 48 pp.
	
\bibitem{La2010}
	Y. Lamzouri,
	{\it  Distribution of values of $L$-functions at the edge of the critical strip},
	Proc. London Math. Soc. (100) \textbf{3}  (2010), 835--863.
	
\bibitem{Le72}
	N. Levinson, {\it $\Omega$-theorems for the Riemann zeta-function }, Acta Arith. \textbf{20} (1972)
	
\bibitem{Li25}
	J. E. Littlewood,  
	{\it On the Riemann zeta-function },
	Proc. London Math. Soc.(2), \textbf{24} (1925), 175--201.
	
\bibitem{LRW2008}
	J. Y. Liu, E. Royer and J. Wu,
	{\it On a conjecture of Montgomery-Vaughan on extreme values of automorphic $L$-functions at $1$},
	Anatomy of integers, CRM Proceedings of Lecture Notes {\bf 46} 
	(American Mathematical Society, Providence, RI, 2008), 217--245.
	
\bibitem{Mon77}
	H.L. Montgomery, {\it Extreme values of the Riemann zeta function}, Comment. Math. Helv. \textbf{52} (1977), 511--518.

\bibitem{Te95}
G. Tenenbaum,
{\it Introduction to analytic and probabilistic number theory}, 
Cambridge Studies in Advanced Mathematics \textbf{46}, Cambridge University Press, Cambridge (1995).
	
\bibitem{Wu2007}
	J. Wu,
	{\it Note on a paper by A. Granville, K. Soundararajan},
	J. Number Theory \textbf{123} (2007), 329--351.
	
\end{thebibliography}
\end{document}